\documentclass{amsart}
\usepackage{a4wide,amssymb}
\usepackage{amsmath,amsthm} 
\usepackage[all]{xy}
\usepackage[english]{babel}
\usepackage[utf8]{inputenc}

\input{diagxy}

\title[On the internal characterization of injective algebras]{On the internal characterization of injective algebras}
\author{Pavlo Dzikovskyi}
\address{P.Dzikovskyi: graduated from Ivan Franko National University of Lviv, Ukraine}
\email{p.dzikovsky@gmail.com}
\subjclass{08A05}
\keywords{injective object}

\newtheorem{proposition}{Proposition}
\newtheorem{lemma}{Lemma}
\newtheorem{remark}{Remark}

\theoremstyle{definition}
\newtheorem{definition}{Definition}

\begin{document}

\begin{abstract} It is shown that universal algebras that are injective in their equational classes are characterized by internal property that can be called completeness. We define universal algebra $A$ as complete (closed to simple extensions) if for each its subalgebra $A'$ and each set of extension conditions for this subalgebra there is $a \in A$ that satisfies these conditions. We define a set of extension conditions for $A'$ to some extension $A''$ as the difference between factorization kernels of free algebras for $A''$ and $A'$. It's proved that each injective universal algebra is complete and each complete universal algebra belonging to the class of algebras with CEP is injective. It's checked directly that complete (in the sense of ordering) boolean algebras and divisible Abelian groups are complete in the sense defined here.
\end{abstract}

\maketitle

\section*{Introduction}

A lot of works is devoted to the injectivity and related topics in universal algebra. In particular essential results in this field (especially in non-categorical aspects of the theory to which we pay special attention) was described in [1], [2], [3]. But perhaps a common internal characterisation of injective algebras has not been made yet. Here we are trying to answer this question.
\\

We use the next notation, common definitions and propositions.
\\

If $\mathbf{M}$ is a variety (equational class) of universal algebras then we denote free in $\mathbf{M}$ algebra of (classes of) words over $S$ as $W(S,\mathbf{M})$ (or just $W(S)$ if it looks unambiguous). Each element of $W(S,\mathbf{M})$ is a factor-class of some word $w$ from the corresponding absolutely free algebra of words $W(S,\Omega)$. We denote it as $[w]$. 
\\

If $\mathbf{M}$ is a variety of universal algebras then any $A \in \mathbf{M}$ is an homomorphic image of some free algebra $F \in \mathbf{M}$. Next consideration often requires to select some corresponding preimage and in each such case we use $W(A,\mathbf{M})$ for the purpose. It does not limit the generality of this consideration because we interprete elements of $A$ in $W(A,\mathbf{M})$ as "pure" sets without operations that are defined on them in $A(\mathbf{M})$.
\\

If homomorphisms $f: A \rightarrow B$ and $g: A \rightarrow B$ coincide on the set of generators $A_0 \subset A$ then they coincide on the whole $A$.
\\

"Let $\mathbf{M}$ be an equational class. An algebra $Q \in \mathbf{M}$ is called injective if for every $B \in \mathbf{M}$, every subalgebra $A$ of $B$ (written $A \le B$) and every $f: A \rightarrow Q$, there exists (a homomorphism) $g: B \rightarrow Q$ extending $A$ (i.e., $g|_A = f$)" [2].
\\

"A class $K$ of algebras is said to satisfy the congruence extension property (CEP) if for all $B \in K$ and all subalgebras $A \le B \in K$, every congruence on $A$ is the restriction of some congruence on $B$" [4].
\\

\section*{Simple extensions of universal algebras}

This section is related to arbitrary variety of universal algebras. We believe that some variety $\mathbf{M}$ is selected and all mentioned algebras belong to this class. So we omit it in the notation below.

\begin{definition}
If $A_0$ is a set of generators of $A$ then we denote the surjective homomorphism (that extends identity map on $A_0$) from $W(A_0)$ on $A$ as $Val_{A_0}$. 
\end{definition}

\begin{definition}
We denote the kernel congruence of $Val_A$ as $Ker_A$. So $([w],[w']) \in Ker_A$ iff $[w]Val_A = [w']Val_A$.
\end{definition}

\begin{definition}
Let $f: S \rightarrow Sf$ is a map. Then we denote the extension of $f$ to surjective homomorphism from $W(S)$ on $W(Sf)$ as $Term_f$.
\end{definition}

\begin{lemma}
Let $h: W(S) \rightarrow A$ is a surjection. Then for each $[w] \in W(S)$: $[[w]Term_{h|_S}]Val_{Sh} = [w]h$, i.e. $Term_{h|_S} \circ Val_{Sh} = h$.
\end{lemma}

\[\bfig
\node ws(250,500)[W(S)]
\node wsh(1000,500)[W(Sh)]
\node a(1000,0)[A]

\arrow[ws`wsh;\,Term_{h|_S}\,]
\arrow|m|[ws`a;\,h\,]
\arrow|m|[wsh`a;\,Val_{Sh}\,]
\efig\]

\begin{proof} 
For each $s \in S$: $[[s]Term_{h|_S}]Val_{Sh} =$ (because $Term_{h|_S}$ coincides with $h$ on $S$) $[[s]h]Val_{Sh} =$ (because $Val_{Sh}$ coincides with identity map on $Sh$) $[s]h$. Thus homomorphisms $Term_{h|_S}\,\circ\,Val_{Sh}$ and $h$ coincide on the set of generators of $W(S)$ and, so, on the whole algebra.
\\
\end{proof} 

\begin{definition}
We call the universal algebra $A'$ a simple extension of $A$ with element $a$ (it's possible that $a \in A$) if $A \le A'$ and $A \cup \{a\}$ is a set of generators of $A'$. We denote it as $A \sqcup a$. 
\end{definition}

\begin{lemma} 
Let $g: A \rightarrow B$ is a surjective homomorphism, $a \notin A$ and $A \sqcup a, B \sqcup b$ are defined. If $g \cup (a,b)$ can be extended to (surjective) homomorphism $g_{\sqcup}: (A \sqcup a) \rightarrow (B \sqcup b)$ then $Term_{g \cup (a,b)} = Term_{h|_{A \cup \{a\}}}$ where $h = Val_{A \sqcup a}\,\circ\,g_\sqcup$.
\end{lemma} 

\[\bfig
\node wAa(0,500)[W(A \cup \{a\})]
\node wBb(2000,500)[W(B \cup \{b\})]
\node Aa(0,0)[A \sqcup a]
\node Bb(2000,0)[B \sqcup b]
\arrow[wAa`wBb;Term_{g \cup (a,b)} = Term_{h|_{A \cup \{a\}}}]
\arrow|m|[wAa`Bb;\,h\,]
\arrow|m|[wAa`Aa;\,Val_{A \sqcup a}\,]
\arrow|m|[wBb`Bb;\,Val_{B \sqcup b}\,]
\arrow[Aa`Bb;\,g_\sqcup]
\efig\]

\begin{proof} 
$h|_{A \cup \{a\}} = g \cup (a,b)$ then both $Term_{g \cup (a,b)}$ and $Term_{h|_{A \cup \{a\}}}$ extend this map on $W(A \cup \{a\}) \rightarrow W(B \cup \{b\})$ and, so, they are equal.
\\ 
\end{proof} 

\begin{definition}
For each $A \sqcup a$ we call $Ker_{A \sqcup a} \setminus Ker_A$ a set of $a$-extension conditions for $A$. We denote it as $Ker_{A \sqcup a}|_a$.
\end{definition}

\begin{proposition} 
Let $g: A \rightarrow B$ is a surjective homomorphism, $a \notin A$ and $A \sqcup a, B \sqcup b$ are defined. We denote $Term_{g \cup (a,b)}$ as $T$. Then $g \cup (a,b)$ can be extended to homomorphism $g_\sqcup: (A \sqcup a) \rightarrow (B \sqcup b)$ iff for each $([w],[w']) \in Ker_{A \sqcup a}|_a$: $([wT],[w'T]) \in Ker_{B \sqcup b}$ (this proposition looks equal to Theorem 2, §12, [5], which is formulated in terms of polynomial symbols).
\end{proposition} 

\[\bfig
\node wAa(0,450)[W(A \cup \{a\})]
\node wBb(1500,450)[W(B \cup \{b\})]
\node Aa(0,0)[A \sqcup a]
\node Bb(1500,0)[B \sqcup b]
\arrow[wAa`wBb;T = Term_{h|_{A \cup \{a\}}}]
\arrow|m|[wAa`Bb;\,h = h_\sqcup\,]
\arrow|m|[wAa`Aa;\,Val_{A \sqcup a} = f_\sqcup \,]
\arrow|m|[wBb`Bb;\,Val_{B \sqcup b}\,]
\arrow[Aa`Bb;\,g_\sqcup]
\efig\]

\begin{proof} 
Necessity. We denote $h = Val_{A \sqcup a}\,\circ\,g_\sqcup$, $Term_{h|_{A \cup \{a\}}} = T$ due to Lemma 2. Let $([w],[w']) \in Ker_{A \sqcup a}|_a$. It follows $([w],[w']) \in Ker_{A \sqcup a}$ that is $[w]Val_{A \sqcup a} = [w']Val_{A \sqcup a}$. So $[wT]Val_{B \sqcup b} = [w]Term_{h|_{A \cup \{a\}}}Val_{B \sqcup b} = $ (due to Lemma 1) $[w]h = ([w]Val_{A \sqcup a})\,g_\sqcup =$ (as mentioned above) $([w']Val_{A \sqcup a})\,g_\sqcup =$ (passing the same steps for $w'$) $[w'T]Val_{B \sqcup b}$ which had to be proved. 
\\

To prove the sufficiency we compare the kernels of homomorphisms $f_\sqcup = Val_{A \sqcup a}$ and $h_\sqcup = T \circ Val_{B \sqcup b}$. Due to Lemma 2: $T|_A = Term_{h_B|_A}$ where $h_B = Val_A\,\circ\,g$. 

\[\bfig
\node wA(0,450)[W(A)]
\node wB(1200,450)[W(B)]
\node A(0,0)[A]
\node B(1200,0)[B]
\arrow[wA`wB;T|_A = Term_{h_B|_A}]
\arrow|m|[wA`B;\,h_B\,]
\arrow|m|[wA`A;\,Val_A\,]
\arrow|m|[wB`B;\,Val_B\,]
\arrow[A`B;\,g]
\efig\]

If $([w],[w']) \in Ker_A$ then $[wT]Val_{B \sqcup b} = [wTerm_{h_B|_A}]Val_B =$ (due to Lemma 1) $[w]h_B = [w]Val_Ag = [w']Val_Ag = $ (passing the same steps for $w'$) $[w'T]Val_{B \sqcup b}$.

And for each $([w],[w']) \in Ker_{A \sqcup a}|_a$: $[wT]Val_{B \sqcup b} = [w'T]Val_{B \sqcup b}$ by the sufficiency condition. 

So for each $([w],[w']) \in Ker_{A \sqcup a}$: $[wT]Val_{B \sqcup b} = [w'T]Val_{B \sqcup b}$ that is $Ker(f_\sqcup) \subset Ker(h_\sqcup)$ and due to the 3rd isomorphism theorem there is homomorphism $g^*: W(A \cup \{a\})f_\sqcup \rightarrow W(A \cup \{a\})h_\sqcup$ such that $f_\sqcup \circ g^* = h_\sqcup$. 

For each $a_0 \in A: a_0g^* = a_0h_\sqcup = a_0g$ and $ag^* = ah_\sqcup = b$. $g^*$ coincides with required $g_\sqcup$ on $A \cup \{a\}$, hence, on the whole $A \sqcup a$. 
\\
\end{proof} 

\begin{proposition} \textbf{["From simple-extension-injectivity follows completeness"]}
Let $B$ is subalgebra of $B_1$. If for any surjective homomorphism $g: A \rightarrow B$ and each simple extension $A \sqcup a$ the homomorphism $g_\sqcup: A \sqcup a \rightarrow B_1$ which extends $g$ can be defined then for each set of extension conditions $Ker_{B \sqcup x}|_x$ there is $b \in B_1$ which satisfies these conditions (that is if $([w(x)],[w'(x)]) \in Ker_{B \sqcup x}|_x$ then $w(b) = w'(b)$).
\end{proposition} 

\begin{proof} 
Let $g$ is an identity map on $B$. For arbitrary $Ker_{B \sqcup x}|_x$ we extend $g$ to $g_\sqcup: B \sqcup x \rightarrow B_1$ which contains $(x,b)$ for some $b \in B_1$. 

Due to the necessity part of Prop.1: $([w(x)T],[w'(x)T]) \in Ker_{B \sqcup b}$, where $T = Term_{g \cup (x,b)}$. Due to Lemma 2: $T = Term_{h|_{B \cup \{x\}}}$, $h = Val_{B \sqcup x}\,\circ\,g_\sqcup$.

So $w(b) = [w(x)]h =$ (due to Lemma 1) $[w(x)]Term_{h|_{B \cup \{x\}}}Val_{B \sqcup b} = [w(x)T]Val_{B \sqcup b}$ \linebreak $= [w'(x)T]Val_{B \sqcup b} =$ (passing the same steps for $w'$) $= w'(b)$ which had to be proved.
\\
\end{proof} 

\begin{proposition} \textbf{["From completeness and CEP follows simple-extension-injectivity"]}
Let $g: A \rightarrow B$ is surjective homomorphism and $A$ belongs to algebraic class that satisfies CEP. Then if $B$ is subalgebra of $B_1$ and for each set of x-extension conditions for $B$ there is $b \in B_1$ which satisfies these conditions (that is if $([w(x)],[w'(x)]) \in Ker_{B \sqcup x}|_x$ then $w(b) = w'(b)$) then for any simple extension $A \sqcup a$ homomorphism $g_{\sqcup b}: A\sqcup a \rightarrow B_1$ which extends $g$ exists.
\end{proposition} 

\begin{proof} 
For arbitrary $A \sqcup a$ the kernel of $g$ can be extended to some congruence relation $ker_{g_\sqcup}$ in $A \sqcup a$ due to CEP. If $g_\sqcup = A \sqcup a \rightarrow (A \sqcup a) / ker_{g_\sqcup}$ then we can assume $(A \sqcup a)g_\sqcup = B \sqcup x$. $x \in B \sqcup x$ satisfies some $x$-extension conditions that include all $T$-images of $a$-extension conditions in $A \sqcup a$ due to the necessity part of Prop.1. By the condition there is $b \in B_1$ that satisfies these $x$-extension conditions so due to the sufficiency part of Prop.1 the required extension of $g$ exists.
\\
\end{proof} 

\begin{remark} 
Another case of extension of $g: A \rightarrow B$ into complete $B_1 \supset B$ can be easy constructed if $g$ is an injection. Then kernel of $g$ can be extended to the trivial congruence relation on $A_1 \supset A$ and it's not necessary $A$ and $A_1$ to satisfy CEP.
\\
\end{remark} 

\section*{Complete (closed to simple extensions) universal algebras}

\begin{definition}
Universal algebra $B_1 \in \mathbf{M}$ is called complete (closed to simple extensions) if for each its subalgebra $B$ and arbitrary set of $x$-extension conditions for $B$ there is $b \in B_1$ which satisfies these conditions (that is if $([w(x)],[w'(x)]) \in Ker_{B \sqcup x}|_x$ then $w(b) = w'(b)$).
\end{definition}

\begin{proposition} 
If universal algebra $B_1$ is injective in $\mathbf{M}$ then it is complete.
\end{proposition} 

\begin{proof} 
Each surjection $g: A \rightarrow B \le B_1$ can be extended to $g_\sqcup: A \sqcup a \rightarrow B_1$ due to the injectivity condition. So required property follows directly from Prop.2.
\\
\end{proof} 

\begin{remark} 
The inversion of Prop.4 is executed in algebraic classes that satisfy CEP due to Prop.3. To construct extension of original homomorphism $A \rightarrow B$ on the whole $A_1 \supset A$ having particular extensions for each $A' \sqcup a'$ in $A_1$ we need to use Zorn's lemma. It's proved in [6], ch.2, th.5.9 (this part of the proof isn't specific for boolean algebras).
\\
\end{remark} 

\begin{remark}
From Prop.4 follows that complete (divisible) Abelian groups and complete boolean algebras are complete in the sense defined above because they are injective in their varieties. But next we check it directly.

Another case of complete algebras we get if $\mathbf{M}$ is a variety of universal algebras such that each $A \in \mathbf{M}$ is free in $\mathbf{M}$ (for example, variety of vector spaces over some field). Then each set of $x$-extension conditions for each $A \in \mathbf{M}$ is empty and $A$ is complete. 
\\
\end{remark}

\begin{proposition} 
Complete Abelian group is complete algebra in the variety of Abelian groups.
\end{proposition} 

\begin{proof} 
Any equality in $Ker_{A \sqcup x}|_x$ in Abelian group can be converted to the form $[x^n] = [a]$ for some $a \in A$. By definition Abelian group is complete (divisible) if for any such equality it contains element that satisfies the equality. So it's a complete algebra.
\\
\end{proof} 

\begin{lemma}
In boolean algebra: $(a \wedge b = a \wedge c) \equiv (a \le (b \Delta c)')$ (here $'$ denotes negation).
\end{lemma}

\begin{proof} 
Necessity. $(a \wedge b = a \wedge c) \Rightarrow (a \wedge b \wedge c' = a \wedge c \wedge c' = 0)$. Similarly $a \wedge c \wedge b' = 0$, so $0 = ((a \wedge b \wedge c') \vee (a \wedge c \wedge b')) = a \wedge ((b \wedge c') \vee (c \wedge b')) = a \wedge (b \Delta c)$. Since $a$ is disjunctive to $b \Delta c$, $a \le (b \Delta c)'$.

Sufficiency. $(a \le (b \Delta c)') \equiv (0 = a \wedge (b \Delta c) = a \wedge ((b \wedge c') \vee (c \wedge b')) = (a \wedge b \wedge c') \vee (a \wedge c \wedge b')) = (a \wedge b \wedge c') = (a \wedge c \wedge b')$. Then $a \wedge b = a \wedge b \wedge (c \vee c') = (a \wedge b \wedge c) \vee (a \wedge b \wedge c') = a \wedge b \wedge c$ and, by the same way, $a \wedge c = a \wedge b \wedge c$. 
\\
\end{proof} 

\begin{lemma}
Any set of $x$-extension conditions for boolean algebra $A$ is equal to a set of inequalities $[x] \ge [w_\alpha] \in W(A)$, $[x] \le [w_\beta] \in W(A)$ where $[w_\alpha] \le [w_\beta]$ for each $\alpha, \beta$.
\end{lemma}

\begin{proof} 
Each boolean term is equal to some term in disjunctive normal form. So each term from $x$ can be converted to the form $(x \wedge b) \vee (x' \wedge c)$ and each $x$-extension condition for $A$ can be written as $[(x \wedge k) \vee (x' \wedge l)] = [(x \wedge m) \vee (x' \wedge n)]$, where $k, l, m, n \in A$.

Since boolean algebra $A$ is isomorphic to $[0,x] \times [0,x']$ for any $x \in A$ (here $[\,]$ denotes closed intervals), each its element is unambiguously projected on the pair of components $[0,x], [0,x']$ and each $x$-extension condition for $A$ is equal to the pair of equalities: $[x \wedge k] = [x \wedge m]$ and $[x' \wedge l] = [x' \wedge n]$, where $k, l, m, n \in A$. Due to Lemma 3 they are equal to $[x] \le [(k \Delta m)']$ and $[x'] \le [(l \Delta n)']$ that is $[x] \ge [l \Delta n]$. 

Because all these inequalities corresponds to the same $x \in A \sqcup x$ each lower bound from them is less or equal to arbitrary upper bound.
\\
\end{proof} 

\begin{proposition} 
Complete boolean algebra is complete algebra in the variety of boolean algebras.
\end{proposition} 

\begin{proof} 
As follows from Lemma 4 each set of $x$-extension conditions for arbitrary subalgebra $A$ of complete boolean algebra $A_1$ is unambiguously determined by some sets of lower and upper bounds for $x$ in $A$. We denote them as $Lower_{A}(x)$ and $Upper_{A}(x)$.  

Consider some set of $x$-extension conditions for subalgebra $A$ of complete boolean algebra $A_1$. Because $A_1$ is a complete lattice the tight bounds of $Lower_{A}(x)$ and $Upper_{A}(x)$ exist in $A_1$ and $sup(Lower_{A}(x)) \le inf(Upper_{A}(x))$. Therefore $a = sup(Lower_{A}(x)) \le inf(Upper_{A}(x)), a \in A_1$ and, so, it satisfies the selected set of $x$-extension conditions. 
\\
\end{proof} 

\begin{remark} 
To prove that injectivity of complete (divisible) Abelian group or complete (in the sense of ordering) boolean algebra follows their completeness (in the sense defined above) due to Prop.3 we need to check that arbitrary homomorphism of Abelian groups or boolean algebras with domain $A$ can be extended on any simple extension $A \sqcup a$. It looks obvious for Abelian groups because any subgroup of Abelian group is normal so it remains a kernel of congruence relation under any simple extension. And as for boolean algebras such extension can be constructed with Lemma 4.
\\
\end{remark}


\begin{thebibliography}{99}

\bibitem{1} Alan Day, "Injectivity in Congruence Distributive Equational Classes", Dissertation, MacMaster University, 1970.

\bibitem{2} Alan Day, "Injectivity in Equational Classes Of Algebras", Can. J. Math., Vol. XXIV, No. 2, 1972, pp. 209-220.

\bibitem{3} David Frank Jacobs, "Amalgamation in Varieties of Algebras", Dissertation, University of Cape Town, 1995.
 
\bibitem{4} Alan Day, "A Note on the Congruence Extension Property", Algebraic Universalis, 11(1971), 234-235.

\bibitem{5} George Grätzer, "Universal Algebra", Second Edition, 1979, Springer Science+Business Media, LLC.

\bibitem{6} Sabine Koppelberg "Handbook of Boolean Algebras"\,, vol.1, North-Holland, 1989.


\end{thebibliography}
\end{document}